\documentclass[12pt]{article}

\usepackage[margin=1in]{geometry}
\usepackage{color}

\usepackage[T1]{fontenc}

\usepackage{lineno,hyperref}
\usepackage{amsmath,amssymb,amsthm}
\usepackage{mathtools}
\usepackage{xcolor}
\usepackage{enumerate}
\usepackage{ifpdf}
\usepackage{algorithmic}
\usepackage{subfigure}
\usepackage{float}
\usepackage{setspace}
\usepackage{appendix}
\usepackage{bm}
\usepackage{multirow}
\usepackage{paralist}
\usepackage{booktabs}
\usepackage{array}
\usepackage{tabularx}
\usepackage{amsfonts}
\usepackage{amsmath}
\allowdisplaybreaks[4]
\usepackage{fancyhdr}
\usepackage{mathrsfs}

\usepackage{dsfont}
\usepackage[linesnumbered,boxed,ruled,commentsnumbered]{algorithm2e}
\usepackage{caption}
\usepackage{color}
\usepackage{multirow}
\usepackage{diagbox}
\usepackage{epstopdf}

\newcommand{\p}{\partial}

\newtheorem{theorem}{Theorem}

\newtheorem{corollary}{Corollary}[section]

\newtheorem{remark}{Remark}

\newcommand{\wt}{\widetilde}

\modulolinenumbers[5]

\begin{document}

\title{Carleman estimate on a tree for Schr\"{o}dinger equation and application to an inverse problem}

\author{Yibin Ding
	\thanks{School of Mathematical Sciences, Zhejiang University, Hangzhou, 310027, China. Email: dybmath@yeah.net}
	\and
	Xiang Xu
	\thanks{School of Mathematical Sciences, Zhejiang University, Hangzhou, 310027, China. Email: xxu@zju.edu.cn}
}
  \date{} 
\maketitle
\begin{abstract}
In this paper, by constructing the weight functions, a global Carleman estimate for the Schr\"{o}dinger equation on a tree is established, with a strong assumption on the solution. And the estimate is able to be applied to derive the Lipschitz stability for an inverse coefficients problem.
\end{abstract}


\section{Introduction}

Inspired by previous work, the Carleman estimate of the transmission line equations on a tree-shaped network, a Carleman estimate for Schr\"{o}dinger equations on a tree-shaped network is given in this work.
For more works about Carleman estimates, we refer readers to \cite{MR2126149,Y001,MR3729280,MR3971248}.
For this method to Schr\"{o}dinger equations, we refer readers to \cite{MR1955903,MR2040524,MR2061430,MR2672251,IPR001,MR3032869,MR3041540}. 
Ignat, Pazoto, and Rosier\cite{IPR001} had given the Carleman estimate for Schr\"{o}dinger equations on a star-shaped network, while their result requires \(N\) groups of weight functions for a \(N\)-edges star-shaped network. For the inverse problem, their result requires observations on all boundary vertexes. 
However, due to their weight functions, their method is hard to reduce the number of weight functions used in the estimate and the number of observation boundary vertexes required in the inverse problem.\par
To get a Carleman estimate for Schr\"{o}dinger equations on a tree-shaped network, the main  difficulty is how to deal with the term derived from inner vertex conditions, which is solved by constructing a group of weight functions and assuming the form of the solution \(u_j=U_j(x) e^{-i\omega t}\) in this work.
And for the inverse problem, a lack of observation on one boundary vertex does not affect the stability result.
This paper is organized as follows. In Section \ref{section.np}, notations are introduced.
In Section \ref{schro.result.sec}, the Carleman estimate of Schr\"{o}dinger equations on a tree-shaped network and the stability of the inverse coefficient problem will be presented.
And in the final section, the main theorem will be proved.
\section{Notation}\label{section.np}
Let \(\Lambda\) be a tree-shaped graph with \(N\) edges \(\{E_j\}_{j\in \mathcal{N}}\) and \(N+1\) vertexes \(\{V_j\}_{j\in\mathcal{N}^0}\), where \(\mathcal{N}=\{1,\dots,N\}\) and \(\mathcal{N}^0=\mathcal{N}\cup\{0\}\).
Without loss of generality, set a boundary vertex, who only connects one edge, as the root vertex \(V_0\). 
For any point \(A\in\Lambda\), set the length of the path from \(V_0\) to itself as its coordinate \(x(A)\). 
If \(V_{k_1}\) and \(V_{k_2}\), where \(k_1,k_2\in\mathcal{N}^0\), are two vertexes of a edge \(E_j,\ j\in\mathcal{N}\) and have \(x(V_{k_1})<x(V_{k_2})\), call \(V_{K_1}\) and \(V_{k_2}\) the initial node \(I_j\) and the terminal node \(T_j\) of \(E_j\) respectively.\par
Let \(S_{I,k}=\{j\in\mathcal{N}:\chi_{V_k}(I_j)=1\}\) be the set of index numbers of edges whose initial node is \(V_k\) and \(S_{T,k}=\{j\in\mathcal{N}:\chi_{V_k}(T_j)=1\}\) be the set of index numbers for edges whose terminal node is \(V_k\). 
Here, \(\chi_{k}(P)=
\begin{cases}
1,\ P\ \text{is}\ V_k,\\
0,\ P\ \text{is not}\ V_k,\\
\end{cases},\ P\in\Lambda,\ k\in\mathcal{N}^0
\).
Let \(\Pi_1=\{k\in\mathcal{N}^0:|S_{I,k}|+|S_{T,k}|=1\}\) be the set of index numbers of boundary vertexes and \(\Pi_2=\{k\in\mathcal{N}^0:|S_{I,k}|+|S_{T,k}|>1\}\) be the set of index numbers of inner vertexes.\par
The interval on an edge \(E_j\) is set as \(\mathcal{I}_j=(x(I_j),x(T_j))\). 
We also introduce a set \(\mathcal{U}(M):=\bigoplus_{j=1}^N \{p\in L^\infty(\mathcal{I}_j);\|p\|_{L^\infty(\mathcal{I}_j)}\leq M\}\), where \(M\) is a constant.
\section{Main result}\label{schro.result.sec}
\subsection{Carleman estimate}
The Schr\"{o}dinger equation on a tree-shaped network is represented as 
\begin{gather}\label{Schro.Original.equation}
i\p_t u_j+\p_x^2 u_j+p_ju_j=f_j,\ j=1,2,\dots,N.
\end{gather}
Here, \(i=\sqrt{-1}\).
On inner vertexes, \(\{u_j\}_{j=1}^N\) meets the Kirchhoff's law:
\begin{gather}\label{Schro.Inner.Cond}
\begin{cases}
u_{j_1}(x(V_k),t)=u_{j_2}(x(V_k),t),\ j_1,j_2\in S_{T,k}\cup S_{I,k},\\
\sum_{j\in S_{T,k}}\p_x u_j(x(V_k),t)=\sum_{j\in S_{I,k}}\p_x u_j(x(V_k),t),
\end{cases}
k\in\Pi_2.
\end{gather}
Firstly, a Carleman estimate of system \eqref{Schro.Original.equation} with the following initial and boundary conditions will be given:
\begin{gather}\label{Schro.Boundary}
\begin{cases}
u_j(x(V_0),t)=0,\ j\in S_{I,0},\ t\in[-T,T],\\
u_j(x(V_k),t)=0,\ j\in S_{T,k},\ k\in\Pi_1\backslash\{0\},\ t\in[-T,T].
\end{cases}
\end{gather}
To get the estimate, set a group of weight functions as \(\varphi_j=\theta\psi_j,\ \theta=\frac{1}{(T-t)(T+t)},\ \psi=a_j x^2+b_j x+c_j,\ j\in\mathcal{N}\) to satisfy following conditions:
\begin{gather}\label{schro.weight.c1}
\begin{cases}
\varphi_{j_1}(x(V_k),t)=\varphi_{j_2}(x(V_k),t),\ j_1,j_2\in S_{T,k}\cup S_{I,k},\\
\p_x \varphi_{j_1}(x(V_k),t)=|S_{I,k}|\p_x \varphi_{j_2}(x(V_k),t),\ j_1\in S_{T,k},\ j_2\in S_{I,k},\\
\p_x^2\varphi_{j_1}(x(V_k),t)=|S_{I,k}|^2\p_x^2\varphi_{j_2}(x(V_k),t),\ j_1\in S_{T,k},\ j_2\in S_{I,k},
\end{cases}
k\in\Pi_2,
\end{gather}
and
\begin{gather}\label{schro.weight.c2}
\varphi_j< 0,\ 
\p_x\varphi_j> 0,\ 
\p_x^2\varphi_j>0,\ 
j=1,2,\dots,N.
\end{gather}

\begin{remark}
A group of weight functions, who meet the above conditions, exist. 
A way to construct this group is introduced and an example will be given.
\par
To set \(a_{j_2},\ b_{j_2},\ c_{j_2},\ {j_2}\in S_{I,k}\), a linear equation can be derived from \eqref{schro.weight.c1} on one inner vertex \(V_k,\ k\in\Pi_2\):
\begin{gather*}
\begin{cases}
a_{j_2} x(V_k)^2+b_{j_2} x(V_k)+c_{j_2}=a_{j_1} x(V_k)^2+b_{j_1} x(V_k)+c_{j_1},\\
|S_{I,k}|(2a_{j_2} x(V_k)+b_{j_2})=2a_{j_1} x(V_k)+b_{j_1},\\
|S_{I,k}|^2 a_{j_2} =a_{j_1},
\end{cases}
j_1\in S_{T,k},\ {j_2}\in S_{I,k}.
\end{gather*}
Thus,  \( a_{j_2}=\frac{a_{j_1}}{|S_{I,k}|^2},\ b_{j_2}=\frac{2(|S_{I,k}|-1)x(V_k)a_{j,1}}{|S_{I,k}|^2}+\frac{b_{j_1}}{|S_{I,k}|},\ c_{j_2}=\frac{(|S_{I,k}|-1)^2x(V_k)^2 a_{j,1}}{|S_{I,k}|^2}+\frac{(|S_{I,k}|-1)x(V_k)b_{j,1}}{|S_{I,k}|}+c_{j,1}\), \(j_1\in S_{T,k},\ j_2\in S_{I,k},\ k\in\Pi_2\). \(-\frac{b_{j_1}}{2a_{j_1}}<0)\) gives \(-\frac{|S_{I,k}|b_{j_1}}{2a_{j_1}}-(|S_{I,k}|-1)x(V_k)<0<x(I_{j_2})\), thus, setting \(-\frac{b_{j}}{2a_{j}}<x(I_{j})=x(V_0)=0,\ j\in S_{I,0}\) is enough to meet the second term of \eqref{schro.weight.c2}.
Obviously, it is easy to select \(c_j<0,\ j\in S_{I,0}\) to meet the first term of \eqref{schro.weight.c2}, which means
\begin{gather*}
\lim_{t\rightarrow -T}\theta^l e^{2s\varphi_j}=\lim_{t\rightarrow T}\theta^l e^{2s\varphi_j}=0,\ \forall l\in \mathbb{N},\ j=1,2,\dots,N.
\end{gather*}
\begin{figure}[!ht]\centering
\includegraphics[width=0.5\textwidth]{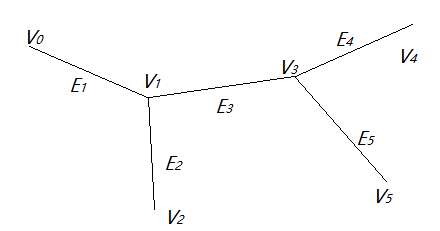}
\caption{An example with 5 edges.}
\label{fig.02}
\end{figure} 
An example is presented on a 5-edges tree in Figure \ref{fig.02} with \(l_j=1,\ j=1,2,\dots,5\). It is easy to set
\begin{gather*}
\begin{cases}
\varphi_1=(x^2+2x-7)\theta,\ x\in\mathcal{I}_1=(0,1),\\
\varphi_2=(\frac{1}{4}x^2+\frac{3}{2}x-\frac{23}{4})\theta,\ x\in\mathcal{I}_2=(1,2),\\
\varphi_3=(\frac{1}{4}x^2+\frac{3}{2}x-\frac{23}{4})\theta,\ x\in\mathcal{I}_3=(1,2),\\
\varphi_4=(\frac{1}{16}x^2+x-4)\theta,\ x\in\mathcal{I}_4=(2,3),\\
\varphi_5=(\frac{1}{16}x^2+x-4)\theta,\ x\in\mathcal{I}_5=(2,3),
\end{cases}
\end{gather*}
which meet conditions \eqref{schro.weight.c1} and \eqref{schro.weight.c2}.
\end{remark}

\begin{theorem}\label{Schro.Carleman.estimate}
Assume that \(\omega\geq 0\), \(f_j=F_j(x) e^{-i( \omega t+\phi)},\ j\in\mathcal{N}\) and \(\{p\}_{j=1}^N\in \mathcal{U}(M)\). If \(\{u_j=U_j(x) e^{-i(\omega t+\phi)}\}_{j=1}^N\in \bigoplus_{j=1}^N H^1(-T,T;H^2(\mathcal{I}_j))\) is a solution of \eqref{Schro.Original.equation} with \eqref{Schro.Inner.Cond}\eqref{Schro.Boundary}, then there exists a constant \(s_0\geq0\), for all \(s>s_0\), we have
\begin{gather*}
\sum_{j=1}^N \int_{Q_j} (s^3\theta^3|u_j|^2+s\theta|\p_x u_j|^2) e^{2s\varphi_j}dxdt
\leq
C\left(\sum_{j=1}^N \int_{Q_j} |f_j|^2 e^{2s\varphi_j}dxdt+B\right),
\end{gather*}
Here, \(Q_j=\mathcal{I}_j\times(-T,T),\ j\in\mathcal{N}\) and
\begin{align*}
B=&\sum_{k\in\Pi_1}\int_{-T}^T\sum_{j\in S_{T,k}}   \theta|\p_x u_j|^2 e^{2s\varphi_j} 
-\sum_{j\in S_{I,k}} \theta |\p_x u_j|^2 e^{2s\varphi_j} dt\Big|_{x=x(V_k)}
\end{align*}
and \(C=C(\Lambda,M,T,s_0,\{\varphi_j\}_{j=1}^N)\) is independent on \(\{{p}_j\}_{j=1}^N,\ s\).
\end{theorem}
The proof is given in the next section.\par
From the definition of \(S_{I,k}\), \(S_{T,k}\) and \(\Pi_1\), we know that \(S_{I,0}\) and \(S_{T,k},\ k\in\Pi_1\backslash\{0\}\) all contain only one number and that \(S_{T,0}\) and \(S_{I,k},\ k\in\Pi_1\backslash\{0\}\) are empty.
It means that 
\begin{align*}
B\leq \sum_{k\in\Pi_1\backslash\{0\}}\int_{-T}^T\sum_{j\in S_{T,k}}   \theta|\p_x u_j|^2 e^{2s\varphi_j} dt\Big|_{x=x(V_k)}.
\end{align*}
Based on the above theorem, it is easy to get the following corollary:
\begin{corollary}\label{Full.Schro.Carleman.estimate}
Assume that \(\omega_m=m\pi/T\geq 0\), \(f_j=\sum_{m=0}^\infty f_{j,m}=\sum_{m=0}^\infty F_{j,m}(x) e^{-i( \omega_m t+\phi_m)},\ j\in\mathcal{N}\) and \(\{p\}_{j=1}^N\in \mathcal{U}(M)\). If \(\{u_j=\sum_{m=0}^\infty u_{j,m}=\sum_{m=0}^\infty U_{j,m}(x) e^{-i(\omega_m t+\phi_m)}\}_{j=1}^N\in \bigoplus_{j=1}^N H^1(-T,T;H^2(\mathcal{I}_j))\) is a solution of \eqref{Schro.Original.equation} with \eqref{Schro.Inner.Cond}\eqref{Schro.Boundary}, then there exists a constant \(s_0\geq0\), for all \(s>s_0\), we have
\begin{gather*}
\sum_{j=1}^N \int_{Q_j} (s^3\theta^3|u_j|^2+s\theta|\p_x u_j|^2) e^{2s\varphi_j}dxdt
\leq
C\left(\sum_{j=1}^N\sum_{m=0}^\infty \int_{Q_j} |f_{j,m}|^2 e^{2s\varphi_j}dxdt+B\right),
\end{gather*}
Here, \(Q_j=\mathcal{I}_j\times(-T,T),\ j\in\mathcal{N}\) and
\begin{align*}
B=&\sum_{k\in\Pi_1\backslash\{0\}}\int_{-T}^T\sum_{j\in S_{T,k}}  \sum_{m=0}^\infty \theta|\p_x u_{j,m}|^2 e^{2s\varphi_j} 
 dt\Big|_{x=x(V_k)}\\
\leq &\frac{1}{T^2}\sum_{k\in\Pi_1\backslash\{0\}}\int_{-T}^T\sum_{j\in S_{T,k}} \sum_{m=0}^\infty  |\p_x u_{j,m}|^2 e^{2s\varphi_j(x,0)}dt\Big|_{x=x(V_k)},\\
=&\frac{1}{T^2}\sum_{k\in\Pi_1\backslash\{0\}}\int_{-T}^T\sum_{j\in S_{T,k}}   |\p_x u_{j}|^2 e^{2s\varphi_j(x,0)}dt\Big|_{x=x(V_k)},\\
\end{align*}
and \(C=C(\Lambda,M,T,s_0,\{\varphi_j\}_{j=1}^N)\) is independent on \(\{{p}_j\}_{j=1}^N,\ s\).
\end{corollary}
\subsection{Inverse problem}
A main application of the Carleman estimate is to get the theoretical stability for the following inverse coefficients problem:\\
{\bf Inverse problem}: Let \(z_j,\ j\in\mathcal{N}\) and \(\phi_k,\ k\in\Pi_1\) be appropriately given. Can we reconstruct the potential \(\{p_j\}_{j=1}^N\) by boundary observation
\begin{gather*}
\p_x u_j(x,t),\ (x,t)\in \{x(V_k)\}\times(0,T),\ j\in S_{I,k}\cup S_{T,k},\ k\in\Pi_1?
\end{gather*}
A stability result can be derived.
Now, consider the following system:
\begin{gather}\label{Schro.ip.sys}
\begin{cases}
i\p_t u_j+\p_x^2 u_j+p_ju_j=0,\ j=1,2,\dots,N,\\
u_j(x,0)=z_j(x),\ x\in \mathcal{I}_j,\ j=1,2,\dots,N,\\
u_j(x(V_k),t)=\phi_k(t),\ j\in S_{I,k}\cup S_{T,k},\ k\in\Pi_1,\ t\in[0,T].
\end{cases}
\end{gather}
Meanwhile, of course, on inner vertexes, \(\{u_j\}_{j=1}^N\) meets \eqref{Schro.Inner.Cond}.
Let \(\{\wt{u}_j\}_{j=1}^N\) be the solution of \eqref{Schro.ip.sys} with potential \(\{\wt{p}_j\}_{j=1}^N\) and \(\{\hat{u}_j\}_{j=1}^N\) be the solution of \eqref{Schro.ip.sys} with potential \(\{\hat{p}_j\}_{j=1}^N\).
In addition, \(\{\wt{u}_j\}_{j=1}^N\) and \(\{\hat{u}_j\}_{j=1}^N\) also have the same form assumed in Corollary \ref{Full.Schro.Carleman.estimate}.
And set \(\bar{u}_j=\wt{u}_j-\hat{u}_j,\ \bar{p}_j=\wt{p}_j-\hat{p}_j\), thus, 
\begin{gather}\label{schro.delta}
\begin{cases}
i\p_t \p_t^m\bar{u}_j+\p_x^2\p_t^m \bar{u}_j+\wt{p}_j\p_t^m\bar{u}_j=-\bar{p}_j\p_t^m\hat{u}_j,\ j=1,2,\dots,N,\\
\p_t^m\bar{u}_j(x(V_k),t)=0,\ j\in S_{I,k}\cup S_{T,k},\ k\in\Pi_1,\ t\in[0,T],
\end{cases}
\end{gather}
whose initial and boundary conditions is 
\begin{gather*}
\begin{cases}
\bar{u}_j(x,0)=0,\\
\p_t\bar{u}_j(x,0)=i\bar{p}_j z_j,\\
\p_t^2\bar{u}_j(x,0)=-\bar{p}_j^2 z_j-\wt{p}_j\bar{p}_j z_j -\p_x^2(\bar{p}_j z_j),\\
\end{cases}
x\in \mathcal{I}_j,\ j=1,2,\dots,N,\\
\begin{cases}
\p_t^m\bar{u}_{j_1}(x(V_k),t)=\p_t^m\bar{u}_{j_2}(x(V_k),t),\ j_1,j_2\in S_{T,k}\cup S_{I,k},\\
\sum_{j\in S_{T,k}}\p_x \p_t^m\bar{u}_j(x(V_k),t)=\sum_{j\in S_{I,k}}\p_x \p_t^m\bar{u}_j(x(V_k),t),
\end{cases}
k\in\Pi_2,\ m=0,1,2.
\end{gather*}
Now, we are ready to state
\begin{theorem}\label{Schro.inverse.po}
We assume that \(|z_j|\geq r>0,\ j\in\mathcal{N}\) and \(\text{Im}\ z_j=0,\ j\in\mathcal{N}\) or \(\text{Re}\ z_j=0,\ j\in\mathcal{N}\).
If \(\{\wt{u}_j\}_{j=1}^N\) is the solution of \eqref{Schro.ip.sys} with potential \(\{\wt{p}_j\}_{j=1}^N\) and \(\{\hat{u}_j\}_{j=1}^N\) is the solution of \eqref{Schro.ip.sys} with potential \(\{\hat{p}_j\}_{j=1}^N\), then there exists a constant \(C=C(\Lambda,M,T,\{z_j\}_{j=1}^N,\{\phi_k\}_{k\in\Pi_1})\),
we have 
\begin{gather*}
\sum_{j=1}^N \int_{\mathcal{I}_j} |\wt{p}_j-\hat{p}_j|^2 dx
\leq C 
\sum_{k\in\Pi_1\backslash\{0\}}\sum_{m=1}^2 \int_{-T}^T\sum_{j\in S_{T,k}}   |\p_t^m\p_x (\wt{u}_j-\hat{u}_j)|^2
\end{gather*},
for any \(\{\wt{p}\}_{j=1}^N,\{\bar{p}\}_{j=1}^N\in \mathcal{U}(M)\).
\end{theorem}
We omit the proof of this theorem, as it is almost the same as the proof of the inverse result of \cite{IPR001}.\par
This theorem gives the stability of the inverse coefficients problem.
As \(V_0\) can be any boundary vertex, a lack of observation on one boundary vertex does not affect the stability result. Of course, the change of \(V_0\) will affect the constant \(C\) in Theorem \ref{Schro.inverse.po}.
\section{Proof of Theorem \ref{Schro.Carleman.estimate}}
In this part, the solution and the source of the original system will be represented as \(u_j=u_{j,1}+i u_{j,2}\) and \(f_j=f_{j,1}+i f_{j,2}\). And set
\(
L_j u_j=
\begin{pmatrix}
\p_x^2&-\p_t\\
\p_t&\p_x^2
\end{pmatrix}
\begin{pmatrix}
u_{j,1}\\
u_{j,2}
\end{pmatrix}
\), then
\begin{gather*}
L_j u_j=
\begin{pmatrix}
\p_x^2&-\p_t\\
\p_t&\p_x^2
\end{pmatrix}
\begin{pmatrix}
u_{j,1}\\
u_{j,2}
\end{pmatrix}
=
\begin{pmatrix}
f_{j,1}-p_j u_{j,1}\\
f_{j,2}-p_j u_{j,2}
\end{pmatrix}.
\end{gather*}

\begin{proof}[Proof of Theorem \ref{Schro.Carleman.estimate}]
Set \(w_j=u_j e^{s\varphi_j}\),
\begin{align*}
P w_j=&(L_j u_{j} )e^{s\varphi_j}=[L_j (w_j e^{-s\varphi_j})]e^{s\varphi_j}\\
=&
\begin{pmatrix}
-\p_t w_{j,2}+s\p_t \varphi_j w_{j,2}+\p_x^2 w_{j,1}-2s\p_x\varphi_j\p_x w_{j,1}+s^2(\p_x\varphi_j)^2 w_{j,1}-s\p_x^2\varphi_j w_{j,1}\\
\p_t w_{j,1}-s\p_t \varphi_j w_{j,1}+\p_x^2 w_{j,2}-2s\p_x\varphi_j\p_x w_{j,2}+s^2(\p_x\varphi_j)^2 w_{j,2}-s\p_x^2\varphi_j w_{j,2}
\end{pmatrix}\\
\triangleq& M_{j,1}+M_{j,2}+M_{j,3},
\end{align*}
where
\(
M_{j,1}=
\begin{pmatrix}
-\p_t w_{j,2}+\p_x^2 w_{j,1}+s^2(\p_x\varphi_j)^2 w_{j,1}\\
\p_t w_{j,1}+\p_x^2 w_{j,2}+s^2(\p_x\varphi_j)^2 w_{j,2}
\end{pmatrix},\ 
M_{j,2}=
\begin{pmatrix}
-2s\p_x\varphi_j\p_x w_{j,1}-s\p_x^2\varphi_j w_{j,1}\\
-2s\p_x\varphi_j\p_x w_{j,2}-s\p_x^2\varphi_j w_{j,2}
\end{pmatrix}\) and
\( 
M_{j,3}=
\begin{pmatrix}
s\p_t \varphi_j w_{j,2}\\
-s\p_t \varphi_j w_{j,1}
\end{pmatrix}
\).
\begin{align}
&\int_{Q_j} |M_{j,1}|^2+|M_{j,2}|^2+2 M_{j,1}^T M_{j,2} dxdt
=\int_{Q_j} |L_j u_j e^{s\varphi_j}-M_{j,3}|^2dxdt\nonumber\\
\leq &C\int_{Q_j} (f_1^2+f_2^2) e^{2s\varphi_j}+p^2(u_1^2+u_2^2)e^{2s\varphi_j}+|M_{j,3}|^2dxdt.\label{sch.ineq01}
\end{align}
\begin{align*}
&2\int_{Q_j}M_{j,1}^T M_{j,2} dxdt\\
=&-2s\int_{Q_j}\p_t w_{j,1}(2\p_x\varphi_j\p_x w_{j,2}+\p_x^2\varphi_j w_{j,2})- \p_t w_{j,2}(2\p_x\varphi_j\p_x w_{j,1}+\p_x^2\varphi_j w_{j,1})dxdt\\
&-2s\int_{Q_j}\p_x^2 w_{j,1}(2\p_x\varphi_j\p_x w_{j,1}+\p_x^2\varphi_j w_{j,1})+\p_x^2 w_{j,2}(2\p_x\varphi_j\p_x w_{j,2}+\p_x^2\varphi_j w_{j,2})dxdt\\
&-2s^3\int_{Q_j}(\p_x\varphi_j)^2[w_{j,1}(2\p_x\varphi_j\p_x w_{j,1}+\p_x^2\varphi_j w_{j,1})+w_{j,2}(2\p_x\varphi_j\p_x w_{j,2}+\p_x^2\varphi_j w_{j,2})]dxdt\\
\triangleq&\sum_{k=1}^3 I_{j,k}.
\end{align*}
Then, estimate \(I_{j,k},\ k=1,2,3\). From
\begin{align*}
&\int_{Q_j} 2\p_x\varphi_j(\p_t w_{j,1}\p_x w_{j,2}-\p_t w_{j,2}\p_x w_{j,1})dxdt\\
=&-\int_{Q_j}\p_x[\p_x\varphi_j(w_{j,1}\p_t w_{j,2}-w_{j,2}\p_t w_{j,1})]dxdt\\
&+\int_{Q_j}\p_x^2\varphi_j(w_{j,1}\p_t w_{j,2}-w_{j,2}\p_t w_{j,1})
-\p_t\p_x\varphi_j(w_{j,1}\p_x w_{j,2}-w_{j,2}\p_x w_{j,1})dxdt,
\end{align*}
we have
\begin{align*}
I_{j,1}=&-2s\int_{Q_j}\p_t w_{j,1}(2\p_x\varphi_j\p_x w_{j,2}+\p_x^2\varphi_j w_{j,2})- \p_t w_{j,2}(2\p_x\varphi_j\p_x w_{j,1}+\p_x^2\varphi_j w_{j,1})dxdt\\
=&-2s\int_{Q_j} 2\p_x\varphi_j(\p_t w_{j,1}\p_x w_{j,2}-\p_t w_{j,2}\p_x w_{j,1})-\p_x^2\varphi_j(w_{j,1}\p_t w_{j,2} -w_{j,2}\p_t w_{j,1})dxdt\\
=&2s\int_{Q_j}\p_x[\p_x\varphi_j(w_{j,1}\p_t w_{j,2}-w_{j,2}\p_t w_{j,1})]
+\p_t\p_x\varphi_j(w_{j,1}\p_x w_{j,2}-w_{j,2}\p_x w_{j,1})dxdt.
\end{align*}
Now, calculate \(I_{j,2}\) and \(I_{j,3}\):
\begin{align*}
I_{j,2}=&-2s\int_{Q_j}\p_x^2 w_{j,1}(2\p_x\varphi_j\p_x w_{j,1}+\p_x^2\varphi_j w_{j,1})+\p_x^2 w_{j,2}(2\p_x\varphi_j\p_x w_{j,2}+\p_x^2\varphi_j w_{j,2})dxdt\\
=&-2s\int_{Q_j}\p_x[\p_x\varphi_j\cdot(|\p_x w_{j,1}|^2+|\p_x w_{j,2}|^2)+\p_x^2\varphi_j\cdot(w_{j,1}\p_x w_{j,1}+w_{j,2}\p_x w_{j,2})]dxdt\\
&+4s\int_{Q_j}\p_x^2\varphi_j\cdot[(\p_x w_{j,1})^2+(\p_x w_{j,2})^2]dxdt,\\
I_{j,3}=&-2s^3\int_{Q_j}(\p_x\varphi_j)^2[w_{j,1}(2\p_x\varphi_j\p_x w_{j,1}+\p_x^2\varphi_j w_{j,1})+w_{j,2}(2\p_x\varphi_j\p_x w_{j,2}+\p_x^2\varphi_j w_{j,2})]dxdt\\
=&-2s^3\int_{Q_j}\p_x[(\p_x\varphi_j)^3(w_{j,1}^2+w_{j,2}^2)]dxdt+4s^3\int_{Q_j}(\p_x\varphi_j)^2\p_x^2\varphi_j\cdot(w_{j,1}^2+w_{j,2}^2)dxdt.
\end{align*}
Sum all \(I_{j,k}\) over \(j\) and \(k\):
\begin{align*}
&2\sum_{j=1}^N\eta_j\int_{Q_j} M_{j,1}^T M_{j,2} dxdt=\sum_{j=1}^N\eta_j\sum_{k=1}^3 I_{j,k}\\
=&\sum_{j=1}^N\eta_j\Big\{
4\int_{Q_j} s\p_x^2\varphi_j(|\p_x w_{j,1}|^2+|\p_x w_{j,2}|^2)+s^3\p_x^2\varphi_j(\p_x\varphi_j)^2(w_{j,1}^2+w_{j,2}^2)dxdt\\
&\hphantom{\sum_{j=1}^N\eta_j\Big\{}
+\int_{Q_j} 2s\p_t\p_x\varphi_j(w_{j,1}\p_x w_{j,2}-w_{j,2}\p_x w_{j,1})dxdt
\Big\}\\
&+\sum_{j=1}^N\eta_j\Big\{
2s\int_{Q_j}\p_x[\p_x\varphi_j(w_{j,1}\p_t w_{j,2}-w_{j,2}\p_t w_{j,1})]dxdt\\
&\hphantom{+\sum_{j=1}^N\eta_j\Big\{}
-2s\int_{Q_j}\p_x[\p_x\varphi_j\cdot(|\p_x w_{j,1}|^2+|\p_x w_{j,2}|^2)]dxdt\\
&\hphantom{+\sum_{j=1}^N\eta_j\Big\{}
-2s\int_{Q_j}\p_x[\p_x^2\varphi_j\cdot(w_{j,1}\p_x w_{j,1}+w_{j,2}\p_x w_{j,2})]dxdt\\
&\hphantom{+\sum_{j=1}^N\eta_j\Big\{}
-2s^3\int_{Q_j}\p_x[(\p_x\varphi_j)^3(w_{j,1}^2+w_{j,2}^2)]dxdt
\Big\}\\
&+\sum_{j=1}^N\eta_j\int_{Q_j} \p_x\Big[
2s\p_x\varphi_j(w_{j,1}\p_t w_{j,2}-w_{j,2}\p_t w_{j,1})
-2s\p_x\varphi_j\cdot(|\p_x w_{j,1}|^2+|\p_x w_{j,2}|^2)\\
&\hphantom{+\sum_{j=1}^N\eta_j\int_{Q_j}\p_x[}
-2s\p_x^2\varphi_j\cdot(w_{j,1}\p_x w_{j,1}+w_{j,2}\p_x w_{j,2})
-2s^3(\p_x\varphi_j)^3(w_{j,1}^2+w_{j,2}^2)
\Big]dxdt\\
=&\sum_{j=1}^N\eta_j\Big\{
4\int_{Q_j} s\p_x^2\varphi_j(|\p_x w_{j,1}|^2+|\p_x w_{j,2}|^2)+s^3\p_x^2\varphi_j(\p_x\varphi_j)^2(w_{j,1}^2+w_{j,2}^2)dxdt\\
&\hphantom{\sum_{j=1}^N\eta_j\Big\{}
+\int_{Q_j} 2s\p_t\p_x\varphi_j(w_{j,1}\p_x w_{j,2}-w_{j,2}\p_x w_{j,1})dxdt
\Big\}\\
&-\sum_{k=0}^N\sum_{m=1}^4 D_{k,m},
\end{align*}
where
\begin{align*}
D_{k,1}=&-2s\int_{-T}^T\sum_{j\in S_{I,k}\cup S_{T,k}}d(k,j)\eta_j\p_x\varphi_j(w_{j,1}\p_t w_{j,2}-w_{j,2}\p_t w_{j,1}) dt\Big|_{x=x(V_k)},\\
D_{k,2}=&2s\int_{-T}^T\sum_{j\in S_{I,k}\cup S_{T,k}}d(k,j)\eta_j\p_x\varphi_j\cdot(|\p_x w_{j,1}|^2+|\p_x w_{j,2}|^2)dt\Big|_{x=x(V_k)},\\
D_{k,3}=
&2s\int_{-T}^T\sum_{j\in S_{I,k}\cup S_{T,k}}d(k,j)\eta_j\p_x^2\varphi_j\cdot(w_{j,1}\p_x w_{j,1}+w_{j,2}\p_x w_{j,2})dt\Big|_{x=x(V_k)},\\
D_{k,4}=
&2s^3\int_{-T}^T\sum_{j\in S_{I,k}\cup S_{T,k}}d(k,j)\eta_j(\p_x\varphi_j)^3(w_{j,1}^2+w_{j,2}^2)
dt\Big|_{x=x(V_k)}.
\end{align*}
Here, \(\eta_j=1,\ j\in S_{I,0},\ \eta_{j_2}=|S_{I,k}|^2 \eta_{j_1},\ j_1\in S_{T,k},\ j_2\in S_{I,k}\), and 
\(d(k,j)=\begin{cases}
1,\ j\in S_{T,k}\\
-1,\ j\in S_{I,k}\\
0,\ \text{others}
\end{cases}
\).\par
The assumption for the solution \(u_j=U_j(x) e^{-i\omega t}\) is the key to deal with \(D_{k,1}\) and the assumption for weights function is going to be used to estimate \(D_{k,2},\ D_{k,3}\), and \(D_{k,4}\).\par
Based on \(w_{j,l}= u_{j,l} e^{s\varphi_j},\ l=1,2\) and condition \eqref{Schro.Inner.Cond}, it is easy to derive
\begin{gather*}
\begin{cases}
w_{j_1,l}(x(V_k),t)=w_{j_2,l}(x(V_k),t),\ j_1,j_2\in S_{T,k}\cup S_{I,k},\\
\p_t w_{j_1,l}(x(V_k),t)=\p_t w_{j_2,l}(x(V_k),t),\ j_1,j_2\in S_{T,k}\cup S_{I,k},\\
\sum_{j\in S_{T,k}} \p_x w_{j,l}(x(V_k),t)=\sum_{j\in S_{I,k}} \p_x w_{j,l}(x(V_k),t),
\end{cases}
l=1,2,\ k\in\Pi_2.
\end{gather*}
{
Assumption \(u_j=U_j e^{-i(\omega t+\phi)}\) gives \(u_{j,1}=U_j \cos(\omega t+\phi)\) and \(u_{j,2}=-U_j\sin(\omega t+\phi)\), in addition, \(\p_t w_{j,l}= (\p_t u_{j,l}+s\p_t \varphi_j u_{j,l})e^{s\varphi_j}\), thus
\begin{align*}
D_{k,1}=&-2s\int_{-T}^T\sum_{j\in S_{I,k}\cup S_{T,k}}d(k,j)\eta_j\p_x\varphi_j(w_{j,1}\p_t w_{j,2}-w_{j,2}\p_t w_{j,1}) dt\Big|_{x=x(V_k)}\\
=&2s\int_{-T}^T \sum_{j\in S_{T,k}}\eta_j(|S_{I,k}|^2-1)\p_x\varphi_j(w_{j,1}\p_t w_{j,2}- w_{j,2}\p_t w_{j,1})dt\Big|_{x=x(V_k)}\\
=&2s\int_{-T}^T \sum_{j\in S_{T,k}}\eta_j(|S_{I,k}|^2-1)\p_x\varphi_j(u_{j,1}\p_t u_{j,2}- u_{j,2}\p_t u_{j,1})e^{2s\varphi_j}dt\Big|_{x=x(V_k)}\\
=&-2s\int_{-T}^T \sum_{j\in S_{T,k}}\eta_j(|S_{I,k}|^2-1)\p_x\varphi_j \omega U_j^2e^{2s\varphi_j}dt\Big|_{x=x(V_k)}\leq 0
\end{align*}
}
As \(|S_{T,k}|=1,\ k\in\Pi_2\), the Schwartz inequality gives
\begin{gather*}
\sum_{j\in S_{T,k}}|\p_x w_{j,l}|^2\leq |S_{I,k}|\sum_{j\in S_{I,k}}|\p_x w_{j,l}|^2,\ l=1,2,
\end{gather*}
which means \(D_{k,2}\leq 0,\ k\in\Pi_2\).
Obviously, \(D_{k,3}=0,\ k\in\Pi_2\).\par
\begin{gather*}
\sum_{j\in S_{T,k}} \eta_j(\p_x\varphi_j)^3-\sum_{j\in S_{I,k}}\eta_j (\p_x\varphi_j)^3=(1-|S_{I,k}|^{-3}\cdot |S_{I,k}|\cdot|S_{I,k}|^2)\sum_{j\in S_{T,k}}\eta_j (\p_x\varphi_j)^3= 0,
\end{gather*}
which gives \(D_{k,4}=0,\ k\in\Pi_2\).\par

Hence, by estimating \(D_{k,m}\), we obtain
\begin{align*}
&\sum_{j=1}^N 2 \int_{Q_j} M_{j,1}^T M_{j,2} dxdt\\
\geq &\sum_{j=1}^N\Big\{
4\int_{Q_j} s\p_x^2\varphi_j(|\p_x w_{j,1}|^2+|\p_x w_{j,2}|^2)+s^3\p_x^2\varphi_j(\p_x\varphi_j)^2(w_{j,1}^2+w_{j,2}^2)dxdt\\
&\hphantom{\sum_{j=1}^N\Big\{}
+\int_{Q_j} 2s\p_t\p_x\varphi_j(w_{j,1}\p_x w_{j,2}-w_{j,2}\p_x w_{j,1})dxdt
\Big\}\\
&+\sum_{j\in S_{I,0}}2s\int_{-T}^T \p_x\varphi_j\cdot(|\p_x w_{j,1}|^2+|\p_x w_{j,2}|^2)dt\Big|_{x=x(V_0)}\\
&-\sum_{k\in \Pi_1\backslash\{0\}}\sum_{j\in S_{T,k}}2s\int_{-T}^T \p_x\varphi_j\cdot(|\p_x w_{j,1}|^2+|\p_x w_{j,2}|^2)dt\Big|_{x=x(V_k)}.
\end{align*}
Take the above inequality into \eqref{sch.ineq01} with a sufficiently large \(s\), the Carleman estiamte of system \eqref{Schro.Original.equation} with condition \eqref{Schro.Inner.Cond}\eqref{Schro.Boundary} can be derived.
\end{proof}

\bibliographystyle{elsarticle-num}

\bibliography{ref}

\end{document}